\documentclass[12pt,a4paper]{article}
\usepackage{indentfirst}
\usepackage{amsmath,amssymb,amsfonts,amsthm,graphics}
\usepackage{makeidx}
\newtheorem{theorem}{Theorem}

\newtheorem{proposition}{Proposition}
\newtheorem{lemma}{Lemma}
\newtheorem{claim}{Claim}
\newtheorem{corollary}{Corollary}

\newtheorem{remark}{Remark}

\newtheorem{conjecture}{Conjecture}
\begin{document}
\title{\Large\bf Rainbow connection of graphs with diameter 2\footnote{Supported by NSFC.}}
\author{\small Hengzhe Li, Xueliang Li, Sujuan Liu\\
\small Center for Combinatorics and LPMC-TJKLC\\
\small Nankai University, Tianjin 300071, China\\
\small lhz2010@mail.nankai.edu.cn; lxl@nankai.edu.cn;
sjliu0529@126.com}
\date{}
\maketitle
\begin{abstract}
A path in an edge-colored graph $G$, where adjacent edges may have
the same color, is called a rainbow path if no two edges of the path
are colored the same. The rainbow connection number $rc(G)$ of $G$
is the minimum integer $i$ for which there exists an
$i$-edge-coloring of $G$ such that every two distinct vertices of
$G$ are connected by a rainbow path. It is known that for a graph
$G$ with diameter 2, to determine $rc(G)$ is NP-hard. So, it is
interesting to know the best upper bound of $rc(G)$ for such a graph
$G$. In this paper, we show that $rc(G)\leq 5$ if $G$ is a
bridgeless graph with diameter $2$, and that $rc(G)\leq k+2$ if $G$
is a connected graph of diameter $2$ with $k$ bridges, where $k\geq
1$.

\end{abstract}

\section{Introduction}

All graphs in this paper are undirected, finite and simple. We refer
to book \cite{bondy} for graph theoretical notation and terminology
not described here. A path in an edge-colored graph $G$, where
adjacent edges may have the same color, is called a $rainbow\ path$
if no two edges of the path are colored the same. An edge-coloring
of graph $G$ is a $rainbow\ edge$-$coloring$ if every two distinct
vertices of graph $G$ are connected by a rainbow path. The $rainbow\
connection\ number\ rc(G)$ of $G$ is the minimum integer $i$ for
which there exists an $i$-edge-coloring of $G$ such that every two
distinct vertices of $G$ are connected by a rainbow path. It is easy
to see that $diam(G)\leq rc(G)$ for any connected graph $G$, where
$diam(G)$ is the diameter of $G$.

The rainbow connection number was introduced by Chartrand et al. in
\cite{char}. It is of great use in transferring information of high
security in multicomputer networks. We refer the readers to
\cite{chak,char2} for details.

Chartrand et al. \cite{char} considered the rainbow connection
number of several graph classes and showed the following proposition
and theorem.

\begin{proposition}\cite{char}  Let $G$ be a nontrivial connected graph of size $m$.
Then

(i) $src(G)=1$ if and only if $G$ is a complete graph;

(ii) $rc(G)=m$ if and only if $G$ is a tree;

(iii) $rc(C_n)=\lceil n/2\rceil$ for each integer $n\geq 4$, where
$C_n$ is a cycle with size $n$.
\end{proposition}

\begin{theorem}\cite{char}
For integers $s$ and $t$ with $2\leq s\leq t,$
$$rc(K_{s,t})=\min\{\lceil\sqrt[s]{t}\rceil, 4\},$$
where $K_{s,t}$ is the complete bipartite graph with bipartition $X$
and $Y$, such that $|X|=s$ and $|Y|=t$.
\end{theorem}

Krivelevich and Yuster \cite{kri} investigated the relation between
the rainbow connection number and the minimum degree of a graph, and
showed the following theorem.

\begin{theorem}\cite{kri}
A connected graph $G$ with $n$ vertices and minimum degree $\delta$
has $rc(G)<\frac{20n}{\delta}.$
\end{theorem}

In fact, Krivelevich and Yuster \cite{kri} made the following
conjecture.

\begin{conjecture}\cite{kri}
If $G$ is a connected graph with $n$ vertices and $\delta(G) \geq
3$, then $rc(G) < \frac{3n}{4}.$
\end{conjecture}

Schiermeyer showed that the above conjecture is true by the
following theorem, and that the following bound is almost best
possible since there exist $3$-regular connected graphs with
$rc(G)=\frac{3n-10}{4}.$

\begin{theorem}\cite{sch}
If $G$ is a connected graph with $n$ vertices and $\delta(G)\geq 3$
then, $$rc(G)\leq \frac{3n-1}{4}.$$
\end{theorem}

Chandran et al. studied the rainbow connection number of a graph by
means of connected dominating sets. A dominating set $D$ in a graph
$G$ is called a $two$-$way\ dominating\ set$ if every pendant vertex
of $G$ is included in $D$. In addition, if $G[D]$ is connected, we
call $D$ a $connected\ two$-$way\ dominating\ set.$

\begin{theorem}\cite{chan}
If $D$ is a connected two-way dominating set of a graph $G$, then
$$rc(G)\leq rc(G[D])+3.$$
\end{theorem}

Let $G$ be a graph. The $eccentricity$ of a vertex $u$, written as
$\epsilon_{G}(u)$, is defined as $\max\{d_{G}(u,v)\,|\;v\in V(G)\}$.
The $radius$ of a graph, written as $rad(G)$, is defined as
$\min\{\epsilon_{G}(u)\;|\;u\in V(G)\}$. A vertex $u$ is called a
$center$ of a graph $G$ if $\epsilon_{G}(u)=rad(G)$.

Basavaraju et al. evaluated the rainbow connection number of a graph
by its radius and {\it chordality (size of a largest induced
cycle)}, and presented the following theorem.

\begin{theorem}\cite{bas}
For every bridgeless graph $G$,
$$rc(G)\leq rad(G)\zeta(G),$$
where $\zeta(G)$ is the size of a largest induced cycle of the graph
$G$.\end{theorem}

They also showed that the above result is best possible by
constructing a kind of tight examples.

Chakraborty et al. investigated the hardness and algorithms for the
rainbow connection number, and showed the following theorem.
\begin{theorem}\cite{chak} Given a graph $G$, deciding if $rc(G)=2$ is NP-Complete. In
particular, computing $rc(G)$ is NP-Hard.
\end{theorem}

It is well-known that almost all graphs have diameter $2$. So, it is
interesting to know the best upper bound of $rc(G)$ for a graph $G$
with diameter 2. Clearly, the best lower bound of $rc(G)$ for such a
graph $G$ is 2. In this paper, we give the upper bound of the
rainbow connection number of a graph with diameter $2$. We show that
if $G$ is a bridgeless graph with diameter $2$, then $rc(G)\leq 5$,
and that $rc(G)\leq k+2$ if $G$ is a connected graph of diameter $2$
with $k$ bridges, where $k\geq 1$.

The end of each proof is marked by a $\Box$. For a proof consisting
of several claims, the end of the proof of each claim is marked by a
$\bigtriangleup$.

\section{Main results}

We need some notations and terminology first. Let $G$ be a graph.
The $k$-$step\ open\linebreak neighbourhood$ of a vertex $u$ in $G$
is defined by $N^k_{G}(u)=\{v\in V(G)\;|\;d_{G}(u,v)=k\}$ for $0\leq
k\leq\ diam(G)$. We write $N_{G}(u)$ for $N^1_{G}(u)$ simply. Let
$X$ be a subset of $V(G)$, and denote by $N^k_{G}(X)$ the set $\{u\
|\ d_{G}(u,X)=k, u\in V(G)\}$, where $d_{G}(u,X)=\min \{d_{G}(u,x)\
|\ x\in X\}$. For any two subsets $X, Y$ of $V(G)$, $E_{G}[X,Y]$
denotes the set $\{xy\;|\; x\in X,y\in Y,xy\in E(G)\}$. Let $c$ be a
rainbow edge-coloring of $G$. If an edge $e$ is colored by $i$, we
say that $e$ is an $i$-$color\ edge$. Let $P$ be a rainbow path. If
$c(e)\in \{i_1,i_2,\ldots,i_r\}$ for any $e\in E(P)$, then $P$ is
called an $\{i_1,i_2,\ldots,i_r\}$-$rainbow\ path$. Let $X_1,
X_2,\ldots\, X_k$ be disjoint vertex subsets of $G$. Notation
$X_1-X_2-\cdots-X_k$ means that there exists some desired rainbow
path $P=(x_1,x_2,\ldots,x_k)$, where $x_i\in X_i,\ i=1,2,\ldots,k$.

\begin{theorem}
Let $G$ be a connected graph of diameter $2$ with $k\geq 1$ bridges.
Then $rc(G)\leq k+2.$
\end{theorem}
\begin{proof}
$G$ must have a cut vertex, say $v$, since $G$ has bridges.
Furthermore, $v$ must be the only cut vertex of $G$, and the common
neighbor of all other vertices due to $diam(G)=2$. Let $G_1,
G_2,\ldots,G_r$ be the components of $G-v$. Without loss of
generality, assume that $G_1, G_2,\ldots,G_k$ are the all trivial
components of $G-v$. We consider the following two cases to complete
this proof.

{\bf Case $1.$} $k=r.$

In this case, we provide each bridge with a distinct color from
$\{1,2,\ldots,k\}$. It is easy to see that this is a rainbow
edge-coloring. Thus $rc(G)\leq k\leq k+2$.

{\bf Case $2.$} $k<r.$

In this case, first provide each bridge with a distinct color, and
denote by $c_1$ this edge-coloring. Next color the other edges as
follows. Let $F$ be a spanning forest of the disjoint union
$G_{k+1}+G_{k+2}+\cdots+G_{r}$ of $G_{k+1},G_{k+2},\ldots,G_{r}$,
and let $X$ and $Y$ be any one of the bipartition defined by this
forest $F$. We provide a $3$-edge-coloring $c_2:
E(G_{k+1}+G_{k+2}+\cdots+G_{r})\rightarrow \{1,k+1,k+2\}$ of $G$
defined by
$$c_2(e)=\left\{\begin{array}{ll} k+1, & if\ e\in E[v,X];\\
 k+2, & if\ e\in E[v,Y];\\ 1, & otherwise.
 \end{array}\right.$$

We show that the edge-coloring $c_1\cup c_2$ is a rainbow
edge-coloring of $G$ in this case. Pick any two distinct vertices
$u$ and $w$ in $V(G)$. If one of $u$ and $w$ is $v$, then $u-w$ is a
rainbow path. If at least one of $u$ and $w$ is a trivial component
of $G-v$, then $u,v,w$ is a rainbow path connecting $u$ and $w$.
Thus we suppose $u,w\in X\cup Y$. If $u\in X$ and $w\in Y$, or $w\in
X$ and $u\in Y$, then $u,v,w$ is a rainbow path connecting $u$ and
$w$. If $u,w\in X$, or $u,w\in Y$, without loss of generality,
assume $u,w\in X$. Pick $z\in Y$ such that $uz\in E[F]$. Thus
$u,z,v,w$ is a rainbow path connecting $u$ and $w$. So $rc(G)\leq
k+2$.

By this all possibilities have been exhausted and the proof is thus
complete.
\end{proof}

\noindent{\bf Tight examples:} The upper bound of Theorem $7$ is
tight. The graph $(kK_1\cup rK_2)\vee {v}$ has a rainbow connection
number achieving this upper bound, where $k\geq 1,r\geq 2$.

\begin{proposition}
Let $G$ be a bridgeless graph with order $n$ and diameter $2$. Then
$G$ is either $2$-connected, or $G$ has only one cut vertex $v$.
Furthermore, $v$ is the center of $G$ with radius $1$.
\end{proposition}
\begin{proof}
Let $G$ be a bridgeless graph with diameter $2$. Suppose that $G$ is
not $2$-connected, that is, $G$ has a cut vertex. Since $diam(G)=2$,
$G$ has only one cut vertex, say $v$. Let $G_1,G_2,\ldots,G_k$ be
the components of $G-v$ where $k\geq 2$. If some vertex, without
loss of generality, say $u\in V(G_1)$, is not adjacent to $v$, then
$d_{G}(u,w)\geq 3$ for any $w\in V(G_2)$. This conflicts with the
fact that $diam(G)=2$. So $v$ is the center of $G$ with radius $1$.
\end{proof}

\begin{lemma}
Let $G$ be a bridgeless graph with diameter $2$. If $G$ has a cut
vertex, then $rc(G)\leq 3.$
\end{lemma}
\begin{remark}
This lemma can be proved by a similar argument for Theorem $7$. It
can also be derived from Theorem $5$.
\end{remark}
\begin{lemma}
Let $G$ be a $2$-connected graph with diameter $2$. Then $rc(G)\leq
5.$
\end{lemma}
\begin{proof} Pick a vertex $v$ in $V(G)$ arbitrarily. Let
$$B=\{u\in N^2_{G}(v)\,|\,there\ exists\ a\ vertex\
w\in N^2_{G}(v)\ such\ that\ uw\in E(G)\}.$$ We consider the
following two cases distinguishing either $B\neq\emptyset$ or
$B=\emptyset$.

{\bf Case $1.$} $B\neq\emptyset$.

In this case, the subgraph $G[B]$ induced by $B$ has no isolated
vertices. Thus there exists a spanning forest $F$ in $G[B]$, which
also has no isolated vertices. Furthermore, let $B_1$ and $B_2$ be
any one of the bipartition defined by this forest $F$. Now divide
$N_{G}(v)$ as follows.

Set $X,Y=\emptyset$. For any $u\in N_{G}(v)$, if $u\in N_G(B_1)$,
then put $u$ into $X$. If $u\in N_G(B_2)$, then put $u$ into $Y$. If
$u\in N_G(B_1)$ and $u\in N_G(B_2)$, then put $u$ into $X$. By the
above argument, we know that for any $x\in X$ ($y\in Y$), there
exists a vertex $y\in Y$ ($x\in X$) such that $x$ and $y$ are
connected by a path $P$ with length $3$ satisfying
$(V(P)\setminus\{x,y\})\subseteq B$.

We have the following claim for any $u\in N_{G}(v)\setminus (X\cup
Y)$.

\begin{claim} Let $u\in N_{G}(v)\setminus(X\cup Y)$. Then
either $u$ has a neighbor $w\in X$, or $u$ has a neighbor $w\in Y$.
\end{claim}
\noindent{\itshape Proof of Claim $1$.} Let $u\in
N_{G}(v)\setminus(X\cup Y)$. Pick $z\in B_1$, then $u$ and $z$ are
nonadjacent since $u\not\in X\cup Y$. Moreover, $diam(G)=2$, so $u$
and $z$ have a common neighbor $w$. We say that $w\not\in
N^2_{G}(v)$, otherwise, $w\in B$ and $u\in X\cup Y$, which
contradicts the fact that $u\not\in X\cup Y$. Moreover, we say that
$w\not\in N_{G}(v)\setminus(X\cup Y)$ by a similar argument. Thus
$w$ must be contained in $X\cup Y$.\hfill$\bigtriangleup$

By the above claim, for any $u\in N_{G}(v)\setminus(X\cup Y)$,
either we can put $u$ into $X$ such that $u\in N_G(Y)$, or we can
put $u$ into $Y$ such that $u\in N_G(X)$. Now $X$ and $Y$ form a
partition of $N_{G}(v)$.

For any $u\in N^2_{G}(v)\setminus B$, let
$$A=\{u\in N^2_{G}(v)\,|\,u\in N_{G}(X)\cap N_{G}(Y)\};$$
$$D_1=\{u\in N^2_{G}(v)\,|\,u\in N_{G}(X)\setminus N_{G}(Y);$$
$$D_2=\{u\in N^2_{G}(v)\,|\,u\in N_{G}(Y)\setminus N_{G}(X)\}.$$
We say that at least one of $D_1$ and $D_2$ is empty. Otherwise,
there exist $u\in D_1$ and $v\in D_2$ such that $d_{G}(u,v)\geq 3$,
which contradicts the fact that $diam(G)=2$. Without loss of
generality, assume $D_2=\emptyset$.

First, we provide a $5$-edge-coloring $c: E(G)\setminus
E_G[D_1,X]\rightarrow \{1,2,\ldots,5\}$ defined by
$$c(e)=\left\{\begin{array}{ll} 1, & if\ e\in E_{G}[v,X];\\
 2, & if\ e\in E_{G}[v,Y];\\ 3, & if\ e\in E_{G}[X,Y]\cup
 E_{G}[Y,A]\cup  E_{G}[B_1,B_2];\\ 4, & if\ e\in E_{G}[X,A]\cup E_{G}[X,B_1]
 ;\\5, & if\ e\in E_{G}[Y,B_2],\ or\ otherwise.
 \end{array}\right.$$

Next, we color the edges of $E_{G}[X,D_1]$ as follows. For any
vertex $u\in D_1$, color one edge incident with $u$ by $5$ (solid
lines), the other edges incident with $u$ are colored by $4$ (dotted
lines). See Figure $1$.

\vspace{12cm}
\scalebox{0.7}[0.7]{\includegraphics[0,0][10,10]{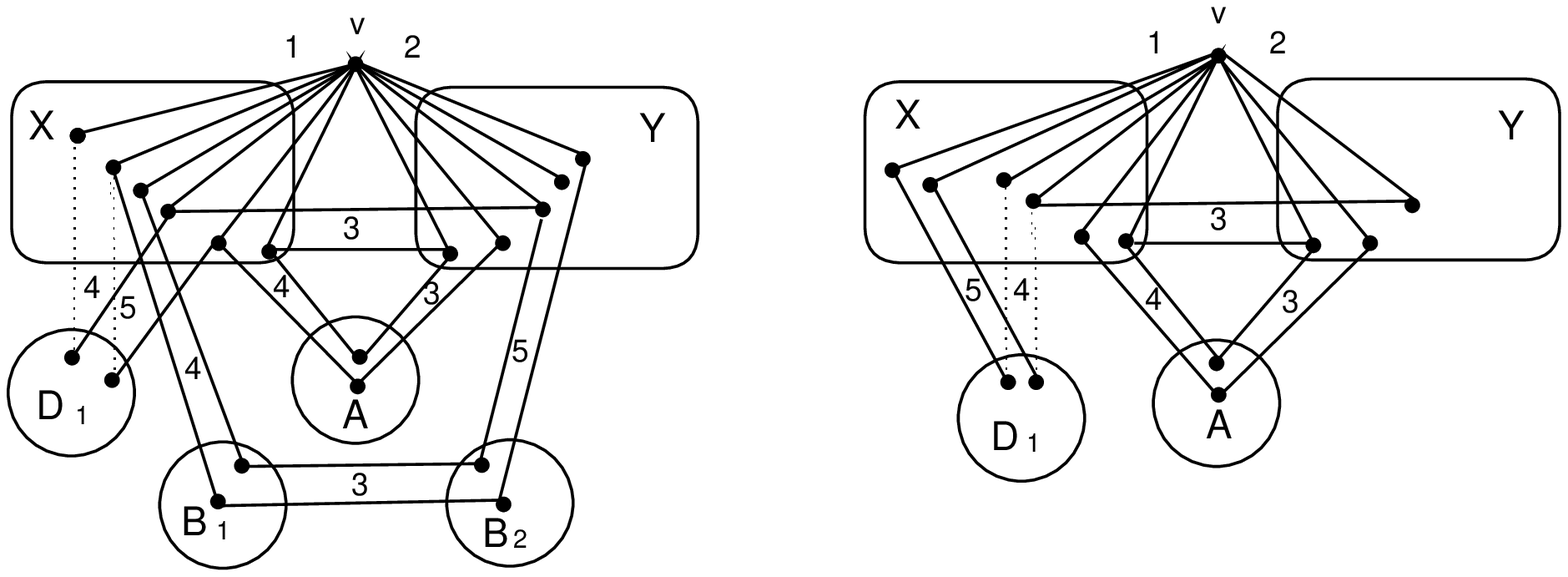}}

\vspace{-7cm} \centerline{Figure $1.$ \hspace{5cm} Figure $2$ }

We have the following claim for the above coloring.
\begin{claim}
$(i)$ For any vertex $x\in X$, there exists a vertex $y\in Y$ such
that $x$ and $y$ are connected by a $\{3,4,5\}$-rainbow path in
$G-v$.

$(ii)$ For any vertex $y\in Y$, there exists a vertex $x\in X$ such
that $x$ and $y$ are connected by a $\{3,4,5\}$-rainbow path in
$G-v$.

$(iii)$ For any $u,u'\in D_1$, there exists a rainbow path
connecting $u$ and $u'$.

$(iv)$ For any $u\in D_1$ and  $u'\in X$, there exists a rainbow
path connecting $u$ and $u'$.
\end{claim}
\noindent{\itshape Proof of Claim $2$.} First, we show that $(i)$
and $(ii)$ hold. We only prove part $(i)$, since part $(ii)$ can be
proved by a similar argument. By the procedure of constructing $X$
and $Y$, we know that for any $x\in X$, either there exists a vertex
$y\in Y$ such that $xy\in E(G)$, or there exists a vertex $y\in Y$
such that $x$ and $y$ are connected by a path $P$ with length $3$
satisfying $(V(P)\setminus\{x,y\})\subseteq B$. Clearly, this path
is a $\{3,4,5\}$-rainbow path.

Next, we show that $(iii)$ holds. Let $u, u'\in D_1$. For any $y\in
Y$, since $diam(G)=2$, we have that $u$ and $y$ have a common
adjacency vertex $w\in X$. Furthermore, without loss of generality,
assume that $uw$ has color $5$. Then $u-w-y-v-w'-u'$ is a rainbow
path connecting $u$ and $u'$, where $u'$ is adjacent to $w'$ by a
$4$-color edge $u'w'$.

Finally, we show that $(iv)$ holds. For any $y\in Y$, since
$diam(G)=2$, we have that $u$ and $y$ have a common adjacency vertex
$w\in X$. Thus $u-w-y-v-u'$ is a rainbow path connecting $u$ and
$u'$. \hfill$\bigtriangleup$

It is easy to see that the above edge-coloring is rainbow in this
case from Figure $1$ and Table $1$.
\begin{figure}[htbp]
{\scriptsize
\begin{center}
\begin{tabular}{|p{0.5cm}|p{0.5cm}|p{1.8cm}|p{1.8cm}|p{1.5cm}|p{1.5cm}|p{2cm}|p{2cm}|}
\hline & $v$ & $X$ & $Y$ & $A$ & $B_1$ & $B_2$ & $D_1$\\\hline $v$ &
--- & $v-X$ & $v-Y$ & $v-X-A$ & $v-X-B_1$ & $v-X-B_1-B_2$ &
$v-X-D_1$
\\\hline
$X$ &--- & Claim $2$ and $Y-v-X$ & $X-v-Y$ & $X-v-Y-A$ &
$X-v-Y-B_2-B_1$& $X-v-Y-B_2$ & Claim $2$
\\\hline
$Y$ & --- & --- & Claim $2$ and $X-v-Y$ & $Y-v-X-A$ & $Y-v-X-B_1$ &
$Y-v-X-B_1-B_2$ & $Y-v-X-D_1$
\\\hline
$A$ & --- & --- & --- & $A-X-v-Y-A$ & $A-Y-v-X-B_1$ & $A-X-v-Y-B_2$
& $A-Y-v-X-D_1$
\\\hline
$B_1$ & --- &--- & --- &---  & $B_1-X-v-Y-B_2-B_1$ & $B_1-X-v-Y-B_2$
& $B_1-B_2-Y-v-X-D_1$
\\\hline
$B_2$ & --- & --- & --- & ---  & --- & $B_2-B_1-X-v-Y-B_2$ &
$B_2-Y-v-X-D_1$
\\\hline
$D_1$ & --- & --- & --- & --- & --- & --- &  Claim $2$
\\\hline
\end{tabular}
\end{center}}
\centerline{Table $1$. The rainbow paths in $G$}
\end{figure}

{\bf Case $2.$} $B=\emptyset$.

In this case, clearly, $N_G(u)\subseteq N_G(v)$ for any $u\in
N^2_G(v)$. To show a rainbow coloring of $G$, we need to construct a
new graph $H$. The vertex set of $H$ is $N_{G}(v)$, and the edge set
is $\{xy\,|\,x,y\in N_G(v),\ x$ and $y$ are connected by a path $P$
with length at most $2$ in $G-v$, and $V(P)\cap N_{G}(v)=\{x,y\}$.

\begin{claim}
The graph $H$ is connected.
\end{claim}
\noindent{\itshape Proof of Claim $3$.} Let $x$ and $y$ be any two
distinct vertices of $H$. Since $G$ is $2$-connected, $x$ and $y$
are connected by a path in $G-v$. Assume that
$P=(x=v_0,v_{1},\ldots,v_{k}=y)$ is a shortest path between $x$ and
$y$ in $G-v$.

If $k=1$, then by the definition of $H$, $x$ and $y$ are adjacent in
$H$. Otherwise, $k\geq 2$. Since $diam(G)=2$, $v_i$ is adjacent to
$v$, or $v_i$ and $v$ have a common neighbor $u_i$ if
$d_{G}(v,v_i)=2$. For any integer $0\leq i\leq k-1$, if
$d_{G}(v,v_i)=1$ and $d_{G}(v,v_{i+1})=1$, then $v_i$ and $v_{i+1}$
are contained in $V(H)$, and adjacent in $H$. If $d_{G}(v,v_i)=1$
and $d_{G}(v,v_{i+1})=2$, then $v_i$ and $u_{i+1}$ are contained in
$V(H)$, and adjacent in $H$. If $d_{G}(v,v_i)=2$ and
$d_{G}(v,v_{i+1})=1$, then $u_i$ and $v_{i+1}$ are contained in
$V(H)$, and adjacent in $H$. If $d_{G}(v,v_i)=2$ and
$d_{G}(v,v_{i+1})=2$, then $u_i$ and $u_{i+1}$ should be contained
in $B$, which contradicts the fact that $B=\emptyset$. Thus, there
exists a path between $x$ and $y$ in $H$. The proof of Claim $3$ is
complete. \hfill$\bigtriangleup$

Let $T$ be a spanning tree of $H$, and let $X$ and $Y$ be the
bipartition defined by this tree. Now divide $N^2_{G}(v)$ as
follows: for any $u\in N^2_{G}(v)$,
$$let\ A=\{u\in N^2_{G}(v)\,|\,u\in N_{G}(X)\cap N_{G}(Y)\};$$
and for any $u\in N^2_{G}(v)\setminus A$,
$$let\ D_1=\{u\in N^2_{G}(v)\,|\,u\in N_{G}(X)\setminus N_{G}(Y)\},$$
$$\quad D_2=\{u\in N^2_{G}(v)\,|\,u\in N_{G}(Y)\setminus N_{G}(X)\}.$$

We say that at least one of $D_1$ and $D_2$ is empty. Otherwise,
there exist $u\in D_1$ and $v\in D_2$ such that $d_{G}(u,v)\geq 3$,
which contradicts the fact that $diam(G)=2$. Without loss of
generality, assume $D_2=\emptyset$. Then $A$ and $D_1$ form a
partition of $N^2_{G}(v)$ (see Figure $2$).

First, we provide a $4$-edge-coloring $c: E(G)\setminus
E_G[D_1,X]\rightarrow \{1,2,\ldots,4\}$ defined by
$$c(e)=\left\{\begin{array}{ll} 1, & if\ e\in E_{G}[v,X];\\
 2, & if\ e\in E_{G}[v,Y];\\ 3, & if\ e\in E_{G}[X,Y]\cup
 E_{G}[Y,A];\\ 4, & if\ e\in E_{G}[X,A],\ or\ otherwise.
 \end{array}\right.$$

Next, we color the edges of $E_{G}[X,D_1]$ as follows. For any
vertex $u\in D_1$, color one edge incident with $u$ by $5$ (solid
lines), the other edges incident with $u$ are colored by $4$ (dotted
lines). See Figure $2$.

Now, we show that the above edge-coloring is a rainbow in this case
from Figure $2$ and Table $2$.

\begin{figure}[htbp]
{\scriptsize
\begin{center}
\begin{tabular}{|p{0.5cm}|p{0.5cm}|p{1.8cm}|p{1.8cm}|p{1.5cm}|p{2cm}|}
\hline & $v$ & $X$ & $Y$ & $A$  & $D_1$\\\hline $v$ & --- & $v-X$ &
$v-Y$ & $v-X-A$  & $v-X-D_1$
\\\hline
$X$ & --- & Claim $2$ and $Y-v-X$ & $X-v-Y$ & $X-v-Y-A$ &
 Claim $2$
\\\hline
$Y$ & --- & --- & Claim $2$ and $X-v-Y$ & $Y-v-X-A$  & $Y-v-X-D_1$
\\\hline
$A$ & --- & --- & --- & $A-X-v-Y-A$ & $A-Y-v-X-D_1$
\\\hline
$D_1$ & --- & --- & --- & ---  & $D_1-A-Y-v-X-D_1$
\\\hline
\end{tabular}
\end{center}}
\centerline{Table $2$. The rainbow paths in $G$}
\end{figure}

By this both possibilities have been exhausted and the proof is thus
complete.
\end{proof}

Combining Proposition $2$ with Lemmas $1$ and $2$, we have the
following theorem.
\begin{theorem}
Let $G$ be a bridgeless graph with diameter $2$. Then $rc(G)\leq 5.$
\end{theorem}

A simple graph $G$ which is neither empty nor complete is said to be
$strongly\ regular$ with parameters $(n,k,\lambda,\mu)$, denoted by
$SRG(n,k,\lambda,\mu)$, if $(i)$ $V(G) = n$; $(ii)$ $G$ is
$k$-regular; $(iii)$ any two adjacent vertices of $G$ have $\lambda$
common neighbors; $(iv)$ any two nonadjacent vertices of $G$ have
$\mu$ common neighbors. It is well known that a strongly regular
with parameters $(n,k,\lambda,\mu)$ is connected if and only if
$\mu\geq 1$.

\begin{corollary}
If $G$ is a strongly regular graph, other than a star, with $\mu\geq
1$, then $rc(G)\leq 5.$

\end{corollary}
\begin{proof}
If $\mu \geq 2$, then $G$ is $2$-connected. Thus $rc(G)\leq 5$ by
Theorem $8$. If $\mu=1$ and $\lambda\geq 1$, then $G$ is bridgeless.
Thus $rc(G)\leq 5$ by Theorem $8$. Thus, the left case is that
$\mu=1$ and $\lambda=0$.

First, suppose that $G$ is a tree. Then $G\cong K_2$ since $G$ is
regular. But this contradicts the fact that $G$ is a strongly
regular graph.

Next, suppose that $G$ is not a tree. We claim that all the induced
cycles of $G$ have length $5$. If $G$ has an induced cycle with
length $3$, then there exist two adjacent vertices $u$ and $v$ in
$C$ such that $|N_{G}(u)\cap N_{G}(v)|\geq 1$, which conflicts with
$\lambda=0$. If $G$ has an induced cycle with length $4$, then there
exist two nonadjacent vertices $u$ and $v$ in $C$ such that
$|N_{G}(u)\cap N_{G}(v)|\geq 2$, which conflicts with $\mu=1$.
Otherwise, $G$ has an induced cycle $C$ with length at least $6$.
Then there exist two nonadjacent vertices $u$ and $v$ in $C$ such
that $|N_{G}(u)\cap N_{G}(v)|=0$, which conflicts with $\mu=1$.

We say that $G$ is bridgeless. By contradiction, let $e=uv$ be a
bridge. Then there exist two components, say $G_1$ and $G_2$, in
$G-v$. Since $G$ is not a tree, there exists a cycle $C$ contained
in $G_1$ (or $G_2$). Without loss of generality, assume that $u\in
V(G_1)$ and $C\subseteq G_1$. Pick $w\in C$ such that $u\not\in
N_{G}(w)$ (There exists such a vertex, since all the induced cycles
of $G$ have length $5$). Then $v$ and $w$ are nonadjacent, and
$N_{G}(v)\cap N_{G}(w)=\emptyset$, which conflicts with $\mu=1$.
Thus $G$ is bridgeless. Therefore $rc(G)\leq 5$ by Theorem $8$.
\end{proof}

\noindent {\bf Remark.} From \cite{chan} we know that the complete
bipartite graph $K_{2, n}$ has a diameter 2, and its rainbow
connection number is 4 for $n\geq 10$. However, we failed to find an
example for which the rainbow connection number reaches 5.

\end{document}